\documentclass[a4paper]{amsart}

\usepackage[usenames, dvipsnames]{xcolor}
\usepackage{amsmath, amssymb}
\usepackage[all]{xy}
\usepackage{tikz}
\usepackage{tikz-cd}
\usepackage{braket}
\usepackage{graphicx}
\usepackage[mathscr]{eucal}

\usepackage[T1]{fontenc}
\usepackage{lmodern}

\usepackage{hyperref}
\usepackage{microtype}

\newtheorem{thm}{Theorem}[subsection]
\newtheorem{lem}[thm]{Lemma}
\newtheorem{coro}[thm]{Corollary}
\newtheorem{prop}[thm]{Proposition}

\theoremstyle{definition}
\newtheorem{defn}[thm]{Definition}
\newtheorem{remark}[thm]{Remark}
\newtheorem{example}[thm]{Example}

\newcommand{\C}{\mathscr{C}}
\newcommand{\h}{\mathrm{h}}
\newcommand{\D}{\mathscr{D}}
\renewcommand{\Set}{\mathrm{Set}}

\newcommand{\op}{\mathrm{op}}
\newcommand{\colim}{\mathrm{colim}}

\title{Adjoint Functor Theorems for $\infty$-categories}	
\author{Hoang Kim Nguyen}
\address{\newline
H.~K.~Nguyen \newline
Fakult\"{a}t f\"ur Mathematik \\
Universit\"{a}t Regensburg \\
D-93040 Regensburg, Germany}
\email{hoang-kim.nguyen@ur.de} 
\author{George Raptis}
\address{\newline
G. Raptis \newline
Fakult\"{a}t f\"ur Mathematik \\
Universit\"{a}t Regensburg \\
D-93040 Regensburg, Germany}
\email{georgios.raptis@ur.de}
\author{Christoph Schrade}
\address{\newline
C. Schrade \newline
Mathematisches Institut \\
WWU M\"{u}nster \\
D-48149 M\"{u}nster, Germany}
\email{cschrad1@uni-muenster.de}
\begin{document}

\begin{abstract}
Adjoint functor theorems give necessary and sufficient conditions for a functor 
to admit an adjoint. In this paper we prove general adjoint functor theorems for functors between $\infty$-categories. One of our main results is an $\infty$-categorical generalization of Freyd's classical General Adjoint Functor Theorem. 
As an application of this result, we recover Lurie's adjoint functor theorems for presentable $\infty$-categories. We also discuss the comparison between adjunctions of $\infty$-categories and homotopy adjunctions, and give a treatment of Brown representability for $\infty$-categories based on Heller's purely categorical formulation of the classical Brown representability theorem.   
\end{abstract}

\maketitle
\setcounter{tocdepth}{1}
\tableofcontents

\section{Introduction}

Adjoint functor theorems give necessary and sufficient conditions for a functor between suitable categories to have an adjoint. They are fundamental results 
in category theory both for their theoretical value as well as for their applications. The most general and well-known adjoint functor theorems are Freyd's \emph{General} and \emph{Special Adjoint Functor Theorem} \cite{Fre, ML}. Other well-known adjoint functor theorems include those specialized to locally presentable categories -- these can also be regarded as useful non-trivial specializations of Freyd's theorems.  

\smallskip 

The purpose of this paper is to prove analogous adjoint functor theorems for functors between $\infty$-categories. Our first main result (Theorem \ref{GAFT}) is an $\infty$-categorical generalization of Freyd's General Adjoint Functor Theorem  and it provides a necessary and sufficient condition, in the form of Freyd's original \emph{solution set condition}, for a limit-preserving functor between $\infty$-categories to admit a left adjoint. In addition, by employing a stronger form of the solution set condition, we find in this higher categorical setting a second and closely related adjoint functor theorem for functors which only preserve finite limits (Theorem \ref{finGAFT}). Both proofs of these theorems are quite elementary, and are based on certain useful criteria for the existence of initial objects, very much in the spirit of the proof of Freyd's classical theorem.   

As in ordinary category theory, the solution set condition can generally be difficult to verify in practice. In ordinary category theory, this condition 
takes a more manageable and simplified form if the categories satisfy additional assumptions. In the $\infty$-categorical setting, as an application of the general adjoint functor theorem, we also obtain specialized adjoint functor theorems, with simplified conditions, in the case of $\infty$-categories which have appropriate additional properties (Theorems \ref{RAFT} and \ref{LAFT}). In particular, these results recover the adjoint functor theorems for presentable 
$\infty$-categories shown by Lurie \cite{HTT}. 

\smallskip

A natural question about adjoint functors in the context of $\infty$-categories is that of comparing adjunctions of $\infty$-categories with the weaker notion of a  functor which becomes an adjoint on the homotopy category. While an adjunction of $\infty$-categories $F \colon \C \rightleftarrows \D \colon G$ induces an adjunction between the homotopy categories $\h F \colon \h \C \rightleftarrows \h \D \colon \h G$, it is false in general that the existence of an adjoint can be detected on the homotopy category. An interesting consequence of our second general adjoint functor theorem is that this converse actually holds when $\D$ admits finite limits and $G$ preserves them (Theorem \ref{finitehomotopyadjunction} -- also Remark \ref{homotopy-adjunction-rem}). 

\smallskip

Another closely related question is that of Brown representability insofar as this can be regarded as a question about adjunctions between homotopy (or other, ordinary) categories. We give a treatment of Brown representability for $\infty$-categories essentially following Heller's purely categorical formulation of the classical Brown representability theorem \cite{He}
(Theorem \ref{Heller}). This approach offers a common unifying perspective for representability theorems (or adjoint functor theorems) for locally presentable categories (Example \ref{loc-pres-cat-2}) and for triangulated categories that arise from stable presentable $\infty$-categories (Theorem \ref{mainTHM} and Corollary \ref{mainTHM-2}). 

\bigskip

The paper is organized as follows. In Section 2, we recall the definition and basic properties of initial objects in an $\infty$-category and establish two criteria 
which ensure the existence of such objects. These criteria (Propositions \ref{finitecriterion} and \ref{initialcriterion}) are based on the weaker 
notions of an \emph{h-initial object} and a \emph{weakly initial set} respectively, and form the technical backbone of the proofs in later sections. 
In Section 3, after a short review of adjoint functors in the higher categorical context, we state and prove two general adjoint functor theorems (Theorems \ref{GAFT} and \ref{finGAFT}). We also discuss the comparison between adjunctions of $\infty$-categories and homotopy adjunctions (leading to Theorem \ref{finitehomotopyadjunction}). In Section 4, we obtain specialized adjoint functor theorems for interesting classes of $\infty$-categories which satisfy additional properties (for example, presentable $\infty$-categories). Since these additional properties are not symmetric, we have in this case different theorems characterizing respectively left and right adjoint functors (Theorems \ref{RAFT} 
and \ref{LAFT}). These theorems are consequences of Theorem \ref{GAFT} and they also recover the adjoint functor theorems for presentable $\infty$-categories in \cite{HTT}. Lastly, in Section 5, we give a treatment of Brown representability for $\infty$-categories following Heller's approach \cite{He}. In particular, this part of the work makes no claim to originality. 
After some preliminaries on Brown representability and its connection with adjoint functor theorems, we present Brown representability theorems for compactly generated $\infty$-categories (Theorem \ref{Heller} -- see Definition \ref{def-comp-gen}) and for stable presentable $\infty$-categories (Theorem \ref{mainTHM}).

\bigskip

\noindent \textbf{Set-theoretical preliminaries.}
We work in a model $\mathbb{V}$ of $\mathrm{ZFC}$-set theory which contains an inaccessible cardinal. We fix the associated Grothendieck universe $\mathbb U \in \mathbb{V}$, which we use to distinguish between small and large sets. More specifically, a set is called \emph{small} if it belongs to $\mathbb U$. Our results do not depend on such set-theoretical assumptions in any essential way, 
but it will be convenient to follow this standard convention for the purpose of
simplifying the exposition.  

A simplicial set is a functor $\Delta^{\op}\to \Set_{\mathbb V}$. We will assume familiarity with the Joyal and the Kan-Quillen model structures on the category
of simplicial sets.
A simplicial set $K \colon \Delta^{\op} \to \Set_{\mathbb V}$ is \emph{small} if $K_n \in \mathbb U$ for each $[n]\in \Delta^{\op}$. An $\infty$-category is \emph{essentially small} if it is (Joyal) equivalent to a small simplicial set. 
An $\infty$-category is called \emph{locally small} if for any small set of objects, the full subcategory that it spans is essentially small (see also \cite[5.4.1]{HTT}). 

An $\infty$-category is (finitely) \emph{complete} (resp. \emph{cocomplete})  if it admits all limits (resp. colimits) indexed by small (finite) simplicial sets. A functor is called (finitely) \emph{continuous} 
(resp. \emph{cocontinuous}) if it preserves all such limits (resp. colimits).

\medskip

\noindent \textbf{Acknowledgements.} We thank Denis-Charles Cisinski for his comments and for his interest in this paper. We also thank Ji\v{r}\'{i} Rosick\'{y} for helpful discussions.  The authors were supported by \emph{SFB 1085 -- Higher Invariants} (Universit\"at Regensburg) funded by the DFG. 

\section{Criteria for the existence of initial objects}

\subsection{Initial objects}
We recall the definition of an initial object in an $\infty$-category and review some of its basic properties. Let $\C$ be an $\infty$-category. An object $x \in \C$ is \emph{initial} if the canonical map
\[
	\C_{x/}\to \C
\]
is a trivial fibration. Equivalently, an object $x \in \C$ is initial if and only if for any $n \geq 1$ and any commutative diagram of the form
\[
\begin{tikzcd}
\Delta^{\{0\}}\drar{x} \dar  & \\
\partial \Delta^n \rar \dar & \C\\
\Delta^n \urar[dashed] &
\end{tikzcd}
\]
there exists an extension as indicated by the dotted arrow \cite[Proposition 4.2]{Joy}. The fiber of the map $\C_{x/} \to \C$ at $y \in \C$ is (a model for) the mapping space $\mathrm{map}_{\C}(x, y)$. Since this map is always a left fibration, it is a trivial fibration if and only if its fibers are contractible Kan complexes (see \cite[Lemma 2.1.3.4]{HTT}). As a consequence, $x \in \C$ is initial if and only if the mapping space $\mathrm{map}_{\C}(x,y)$ is contractible for all $y \in \C$. The full subcategory of 
$\C$ which is spanned by the initial objects is either empty or a contractible Kan 
complex (see \cite[Proposition 4.4]{Joy} or \cite[1.2.12.9]{HTT}). An object $x \in \C$ is \emph{terminal} if it defines an initial object in the opposite $\infty$-category $\C^{\op}$. Following \cite{Joy}, a \emph{limit} of a diagram $p \colon K \to \C$, where $K$ is a simplicial set and $\C$ is an $\infty$-category, is by definition a terminal object of the slice $\infty$-category $\C_{/p}$, that is, a terminal cone $K^\triangleleft:=\Delta^0 \star K \to \C$ over the diagram $p$. We will often refer to the evaluation at the cone-object as the limit of the diagram. 

\medskip

We will need the following well-known characterization of initial objects in an $\infty$-category. For completeness, we include an elementary proof based on the definitions given above (see also \cite[Proposition 4.2]{Joy} and \cite[Lemma 4.2.3]{RV} for closely related results).

\begin{prop}\label{idlimit}
Let $\C$ be an $\infty$-category. Then $x  \in \C$ is an initial object if and only if the identity functor $\mathrm{id}: \C \to \C$ admits a limit whose 
cone-object is $x \in \C$. 
\end{prop}
\begin{proof}
First note that for any pair of cones $\gamma, \delta: \Delta^0 \star \C \to \C$ over the identity functor, there is a canonical morphism of cones $\gamma \to \delta$ which is given by
\[
\Delta^1 \star \C \cong \Delta^0 \star \Delta^0 \star \C \xrightarrow{\mathrm{id} \star \delta} \Delta^0 \star \C \xrightarrow{\gamma} \C.
\]
Suppose that there is a limit cone over the identity functor 
$$\lambda: \Delta^0 \star \C \to \C$$
with cone-object $\lambda(\ast) = \colon x \in \C$. Then we obtain a canonical endomorphism of cones $\phi: \lambda \to \lambda$ as explained above. Since $\lambda$ is a terminal object in the category of cones over the identity functor, this morphism is an equivalence in $\C_{/\mathrm{id}}$. In particular, the evaluation of this morphism at the cone-object
$\phi(\ast): x \to x$ is an equivalence in $\C$. We need to show that each commutative diagram
\begin{equation} \label{diagram-1}
\begin{tikzcd}
\Delta^{\{0\}} \drar{x}\dar & \\ 
\partial \Delta^n \rar \dar & \C\\
\Delta^n \urar[dashed] &
\end{tikzcd}
\end{equation}
admits an extension as indicated by the dotted arrow. Applying $\Delta^0 \star (-)$ to this diagram and then composing with the cone $\lambda$, we obtain  
a new diagram as follows
\[
\begin{tikzcd}
\Delta^{\{0,1\}} \drar{\phi(\ast)}\dar & \\ 
\Lambda^{n+1}_{0}\rar \dar & \C.\\
\Delta^{n+1} \urar[dashed] &
\end{tikzcd}
\]
Since $\phi(\ast)$ is an equivalence, it follows from \cite[Theorem 2.2]{Joy} that the required extension exists. The restriction of this extension along the inclusion $\Delta^n \subset \Delta^0 \star \Delta^n = \Delta^{n+1}$ gives the required extension for the original diagram \eqref{diagram-1}. 

Conversely, suppose that $x \in \C$ is an initial object so that the map $\C_{x/}\to \C$ is a trivial fibration. First we find a cone over the identity with cone-object $x \in \C$ as a solution of the following lifting problem
\[
\begin{tikzcd}
\Delta^0 \rar{1_x} \dar[swap]{x} & \C_{x/}\dar \\
\C \rar[swap]{\mathrm{id}} \urar[dashed]{\lambda} & \C.
\end{tikzcd}
\]
We claim that $\lambda$ defines a terminal object in $\C_{/ \mathrm{id}}$. For this, it suffices to show that for each commutative diagram, $n \geq 1$,
\begin{equation} \label{diagram-2}
\begin{tikzcd}
\Delta^{\{n\}}\star \C \drar{\lambda} \dar & \\
\partial \Delta^n \star  \C \rar \dar & \C\\
\Delta^n\star \C \urar[dashed] &
\end{tikzcd}
\end{equation}
there is an extension as indicated by the dotted arrow. Here we have used the same notation $\lambda$ for the map which is adjoint to the lift above. We extend this diagram to a new diagram as follows
\[
\begin{tikzcd}
\Delta^1 \star \C \cong \Delta^{\{n\}} \star \Delta^0 \star \C \rar{\mathrm{id} \star \lambda}\dar & \Delta^{\{n\}}\star \C \drar{\lambda} \dar & \\
\Lambda^{n+1}_{n+1}\star \C\cong\partial \Delta^n \star \Delta^0 \star \C \rar{\mathrm{id} \star \lambda}\dar & \partial \Delta^n \star \C \rar \dar & \C.\\
\Delta^{n+1}\star \C \cong \Delta^n \star \Delta^0 \star \C \rar{\mathrm{id} \star \lambda} &\Delta^n \star \C \urar[dashed] &
\end{tikzcd}
\]
By adjunction, the composite extension problem corresponds to finding an extension in the following diagram 
\begin{equation} \label{diagram-3}
\begin{tikzcd}
\Delta^{\{n,n+1\}} \drar \dar & \\
\Lambda^{n+1}_{n+1} \rar \dar & \C_{/\mathrm{id}}.\\
\Delta^{n+1}\urar[dashed] &
\end{tikzcd}
\end{equation}
Note that by construction the morphism $\Delta^{\{n,n+1\}} \to \C_{/\mathrm{id}}$ is an endomorphism of the cone $\lambda$. This is an equivalence since the underlying morphism on cone-objects is the identity of $x$ and $\C_{/\mathrm{id}} \to \C$ is conservative (as a right fibration). Thus, again by \cite[Theorem 2.2]{Joy}, there exists an extension in \eqref{diagram-3} as required. The adjoint map of this extension restricts along $\Delta^n \subset \Delta^n \star \Delta^0 = \Delta^{n+1}$ to an extension for the original diagram \eqref{diagram-2}.
\end{proof}

\subsection{$h$-initial objects}

We consider the following weakening of the notion of an initial object in an $\infty$-category $\C$. 

\begin{defn}
Let $\C$ be an $\infty$-category. An object $x \in \C$ is \emph{h-initial} if it defines an initial object in the homotopy category $\h \C$. Dually, an object is called \emph{h-terminal} if it defines a terminal object in $\h\C$.
\end{defn}

We refer to \cite{Joy}, \cite{HTT} for the construction and basic properties of the homotopy category. Note that an object $x \in \C$ is $h$-initial if and only if the mapping space $\mathrm{map}_{\C}(x, y)$ is non-empty and connected for all $y \in \C$. This property can also be stated in terms of extension problems as follows: $x \in \C$ is $h$-initial if and only if for any $1 \leq n \leq 2$ and any commutative diagram of the form
\[
\begin{tikzcd}
\Delta^{\{0\}}\drar{x} \dar  & \\
\partial \Delta^n \rar \dar & \C\\
\Delta^n \urar[dashed] &
\end{tikzcd}
\]
there exists an extension as indicated by the dotted arrow. Clearly an initial object is also $h$-initial, but the converse is false in general. The following result shows the converse under appropriate assumptions on $\C$.

\begin{prop}\label{finitecriterion}
Let $\C$ be an $\infty$-category which admits finite limits. Then an object $x \in \C$ is $h$-initial if and only if it is initial.
\end{prop}
\begin{proof}
Suppose that $x \in \C$ is $h$-initial. Then for any object $y\in \C$ the mapping space $\mathrm{map}_{\C}(x,y)$ is non-empty and connected. Since $\C$ admits finite limits, for any object $y \in \C$ and any finite simplicial set $K$, there exists an object $y^K \in \C$ such that there is a natural isomorphism in the homotopy category of spaces,
\[
\mathrm{map}_{\C}(x,y^K) \cong \mathrm{map}_\mathcal S(K,\mathrm{map}_{\C}(x,y))
\]
where $\mathcal S$ denotes the $\infty$-category of spaces. (See \cite[Corollary 4.4.4.9]{HTT}.) In particular, since $x$ is $h$-initial, these mapping spaces are non-empty and connected for any finite simplicial set $K$. It follows that $\mathrm{map}_{\C}(x, y)$ is contractible for any $y \in \C$, and hence $x$ is an initial object. 
\end{proof}

\subsection{Weakly initial sets} We consider also a further weakening of the notion of an initial object in an $\infty$-category. 

\begin{defn}
Let $\C$ be an $\infty$-category. A set of objects $S \subset \C$ is \emph{weakly initial} if for any object $y \in \C$ there is $x_s \in S$ such that the mapping space $\mathrm{map}_{\C}(x_s, y)$ is non-empty. Dually, $S$ is \emph{weakly terminal} if it defines a weakly initial set of objects in $\mathscr C^{\op}$, that is, if for any $y\in \C$ there exists $x_s\in S$ such that  the mapping space $\mathrm{map}_\C(y,x_s)$ is non-empty.
\end{defn}

Clearly an $h$-initial object defines a weakly initial set consisting of a single object, but a weakly initial set which consists of a single object is 
not $h$-initial in general. The property that the singleton set $\{x\} \subset \C$ is weakly initial (i.e., $x$ is a \emph{weakly initial object}) says that for any commutative diagram of the form
\[
\begin{tikzcd}
\Delta^{\{0\}}\drar{x} \dar  & \\
\partial \Delta^1 \rar \dar & \C\\
\Delta^1 \urar[dashed] &
\end{tikzcd}
\]
there exists an extension as indicated by the dotted arrow. If $S \subset \C$ is a 
weakly initial set which is small and $\C$ admits small products, then the object $\prod_{x_s \in S} x_s \in \C$ is a weakly initial object in $\C$. 

The following result shows that the existence of a weakly initial set implies the existence of an initial object under appropriate assumptions on $\C$. A related result can be found in \cite[Lemma 3.2.6]{BHH}.

\begin{prop}\label{initialcriterion}
Let $\C$ be an $\infty$-category which is locally small and complete. Then $\C$ admits an initial object if and only if it admits a small weakly initial set.
\end{prop}

We give a proof in Subsection \ref{proof-of-initial-criterion} after we recall a few facts about coinitial functors, i.e., the dual notion of a cofinal functor. 

\subsection{Reminder about coinitial functors} We recall the definition of a coinitial functor and review some of its main properties following and dualizing \cite[4.1]{HTT}. This notion is required for the proof of Proposition \ref{initialcriterion} and can
also be used to reformulate the different notions of initiality for objects in an $\infty$-category as defined earlier.  

\smallskip

Given simplicial sets $A$ and $B$, we denote by $\mathrm{hom}(A,B)$ the internal hom-object with respect to the cartesian product of simplicial sets. Suppose furthermore that we have structure maps $A\to T$ and $B\to T$, then we denote by $\mathrm{hom}_{/T}(A,B)\subseteq \mathrm{hom}(A,B)$ the simplicial subset of those maps which commute with the structure maps.

\begin{defn}
A map of simplicial sets $S\to T$ is \textit{coinitial} if for any left fibration $X \to T$ the induced map
\[
\mathrm{hom}_{/T}(T,X)\to \mathrm{hom}_{/T}(S,X)
\]
is a homotopy equivalence of simplicial sets.
\end{defn}

The following proposition collects some properties of coinitial maps. 

\begin{prop} \label{final-maps}
\begin{enumerate}
\item[(1)] A map $f \colon K \to K'$ is coinitial if and only if for any $\infty$-category 
and any diagram $p \colon K' \to \C$, the induced map 
$$\C_{/p} \to \C_{/q}$$
is an equivalence of $\infty$-categories, where $q = p \circ f$.
\item[(2)]  A map $f \colon K \to K'$ is coinitial if and only if for any $\infty$-category $\C$ and any limit cone $p \colon K'^{\triangleleft} \to \C$, the induced diagram $$K^\triangleleft \to K'^\triangleleft \xrightarrow{p} \C$$
is also a limit cone. 
\end{enumerate}
\end{prop}
\begin{proof}
See \cite[Proposition 4.1.1.8]{HTT}.  
\end{proof}

A useful recognition theorem for coinitial maps is given by the following generalization of Quillen's Theorem A.  We first introduce some notation: given 
a map $F: K \to \D$ between simplicial sets and an object $d\in \D$, we denote by $F_{/d}$ the pullback of simplicial sets
\[
\begin{tikzcd}
F_{/d} \rar \dar & \D_{/d}\dar \\
K \rar{F} & \D
\end{tikzcd}
\]
and dually, by $F_{d/}$, the corresponding pullback along the projection 
$\D_{d/}\to \D$.

\begin{thm}[Joyal] \label{QuillenA}
Let $F\colon \C \to \D$ be a functor between $\infty$-categories. Then $F$ is coinitial if and only if $F_{/d}$ is weakly contractible for every object $d \in \D$.
\end{thm}
\begin{proof}
See \cite[Theorem 4.1.3.1]{HTT}.
\end{proof}

As an immediate consequence of this characterization, we note that an object $x \in \C$ is initial if and only if the map $\Delta^0 \xrightarrow{x} \C$ is coinitial. Similarly, the notions of an $h$-initial object and a weakly initial set can be reformulated in terms of the following weaker notions of coinitiality. We say that a map
$F \colon K \to \D$ between simplicial sets is \emph{h-coinitial} (resp. \emph{weakly coinitial}) 
if the simplicial set $F_{/d}$ is non-empty and connected (resp. non-empty) for every $d \in \D$. Then we obtain the following obvious reformulations of our previous definitions. 

\begin{prop}
Let $\C$ be an $\infty$-category. 
\begin{enumerate}
\item[(1)] An object $x \in \C$ is $h$-initial if and only if $\Delta^0 \xrightarrow{x} \C$ is h-coinitial. 
\item[(2)] A set of objects $S \subset \C$ is weakly initial if and only if the inclusion 
(of objects) $S \hookrightarrow \C$ is weakly coinitial. 
\end{enumerate}
\end{prop}

\subsection{Proof of Proposition \ref{initialcriterion}} \label{proof-of-initial-criterion} We first state a few elementary lemmas.

\begin{lem}\label{slicecomplete}
Let $\C$ be an $\infty$-category and let $c \in \C$ be an object. If $\C$ is complete, then so is $\C_{/c}$.
\end{lem}
\begin{proof}
Let $f \colon K \to \C_{/c}$ be a diagram and $f' \colon K \star \Delta^0 \to \C$ its adjoint. The adjoint of a limit cone $(K \star \Delta^0)^\triangleleft 
\to \C$ for $f'$ defines a limit cone for $f$. 
\end{proof}

\begin{lem}\label{conemorphism}
Let $\C$ be an $\infty$-category, $x \in \C$ an object, and let $\lambda:\Delta^0 \star K \to \C$ be a cone. Then a morphism $u \colon x \to \lambda(\ast)$ determines a cone $\lambda': \Delta^0 \star K \to \C$ with cone-object $\lambda'(\ast)= x$ and 
a morphism of cones $\phi \colon \lambda' \to \lambda$ with $\phi(\ast) =  u$.
\end{lem}
\begin{proof}
The morphism $x \to \lambda(\ast)$ determines a map $\Delta^1 \cup_{\Delta^0}\Delta^0 \star K \to \C$. Then the result follows by choosing an extension along the 
inclusion $\Delta^1\cup_{\Delta^0}\Delta^0 \star K \subseteq \Delta^1 \star K$, 
which is inner anodyne by \cite[Lemma 2.1.2.3]{HTT}.
\end{proof}

\noindent \textbf{Proof of Proposition \ref{initialcriterion}.} It is clear that $\C$ admits a small weakly initial set if it admits an initial object. 

Conversely, suppose that $S \subset \C$ is a small weakly initial set. By Lemma \ref{idlimit}, we need to show that the identity functor $\mathrm{id}: \C \to \C$ admits a limit. Note that $\C$ is not assumed to be small in general. Let $\C'$ denote the full subcategory of $\C$ that is spanned by $S$. We 
claim that the inclusion $U \colon \C' \hookrightarrow \C$ is a coinitial functor.  
By Theorem \ref{QuillenA}, it suffices to show that the $\infty$-category $\C'_{/c} : = U_{/c}$ is weakly contractible for each $c \in \C$. 

Let $K$ be a small simplicial set and let $\lambda: K \to \C'_{/c}$ be a map. We consider the composition
\[
\mu: K \xrightarrow{\lambda} \C'_{/c} \to \C_{/c}.
\]
Since $\C$ is complete, so is $\C_{/c}$ by Lemma \ref{slicecomplete}. Hence there is an extension to a limit cone as follows
\[
\begin{tikzcd}
K \rar \dar & \C'_{/c} \rar & \C_{/c}.\\
K^{\triangleleft} \arrow{urr}[swap]{\overline \mu} & &
\end{tikzcd}
\]
The cone-object $\overline{\mu}(\ast)$ corresponds to a morphism $(l \to c)$ in $\C$. Since $\C' \subset \C$ is weakly coinitial, there is an object $c'\in \C'$ and a morphism $\gamma: c' \to l$. This determines a morphism $(\overline{\mu}(\ast) \circ\gamma\to \overline{\mu}(\ast)) $ in $\C_{/c}$. By Lemma \ref{conemorphism}, this last morphism extends to a morphism of cones
\[
\Gamma: \Delta^1 \star K \to \C_{/c}
\]
such that $\Gamma|_{\Delta^{\{1\}}\star K}= \overline{\mu}$ and $\Gamma|_{\Delta^{\{0\}}}= \overline{\mu}(\ast) \circ \gamma$. Let $\Gamma_0:= \Gamma|_{\Delta^{\{0\}}\star K}$ and consider the composition
\[
\Delta^0 \star K \xrightarrow{\Gamma_0} \C_{/c}\to \C.
\]
We observe that this composition sends every vertex of $\Delta^0 \star K$ to a vertex belonging to $\C'$. Since $\C'$ is a full subcategory, the functor $\Gamma_0$ factors through the inclusion $\C'_{/c} \subset \C_{/c}$,
\[
\begin{tikzcd}
  & \C'_{/c}\dar \\
\Delta^0 \star K \rar{\Gamma_0} \urar[dashed]{\Gamma'_0} & \C_{/c}.
\end{tikzcd}
\]
By construction, $\Gamma'_0$ extends $\lambda: K \to \C'_{/c}$. In conclusion, any map $K \to \C'_{/c}$ admits an extension as follows
\[
\begin{tikzcd}
K \rar \dar & \C'_{/c}.\\
K^\triangleleft \urar[dashed] &
\end{tikzcd}
\]
Since $K^\triangleleft$ is weakly contractible, it follows by standard arguments that $\C'_{/c}$ is weakly contractible and therefore $\C' \hookrightarrow \C$ is coinitial, as claimed. Moreover the diagram $U\colon \C' \to \C$ admits a limit because $\C$ admits small limits and $\C'$ is essentially small by our assumptions on $\C$. Using Proposition \ref{final-maps}(1), this implies that the identity $\mathrm{id}\colon \C \to \C$ also admits a limit. Then $\C$ has 
an initial object by Proposition \ref{idlimit}. \qed

\section{General adjoint functor theorems}

\subsection{Recollections about adjoint functors} We recall the definition of an adjunction between $\infty$-categories and review some of its main properties
following the treatment in \cite[5.2]{HTT}. Given a map $q \colon \mathscr{M} \to \Delta^1$ we write $\mathscr{M}_0$ (resp. $\mathscr{M}_1$) for the fiber at the 
$0$-simplex $\Delta^{\{0\}}\to \Delta^1$ (resp. $\Delta^{\{1\}} \to \Delta^1$). We say that $q$ is a \emph{bicartesian fibration} if it is both a cartesian and a cocartesian fibration.  

\begin{defn}
Let $\C$ and $\D$ be $\infty$-categories. An \emph{adjunction} between $\C$ and $\D$ consists of a bicartesian fibration $q \colon \mathscr{M} \to \Delta^1$ together with equivalences $\C \simeq \mathscr{M}_0$ and $\D\simeq \mathscr{M}_1$. 
\end{defn}

\noindent An adjunction $\mathscr{M} \to \Delta^1$  determines essentially uniquely functors as follows
$$F \colon \C \simeq \mathscr{M}_0 \to \mathscr{M}_1 \simeq \D$$
$$G \colon \D \simeq \mathscr{M}_1 \to \mathscr{M}_0 \simeq \C.$$
Then we say that $F$ is left adjoint to $G$ (resp. $G$ is right adjoint to $F$). Conversely, given a pair of functors $F \colon \C \rightleftarrows \D \colon G$, $F$ is left adjoint to $G$ if and only if there is a natural transformation 
$$u \colon \mathrm{id}_{\C} \to G \circ F$$
such that the composition 
$$\mathrm{map}_{\D}(F(c), d) \stackrel{G}{\to} \mathrm{map}_{\C}(G(F(c)), G(d)) \stackrel{u^\ast}{\to} \mathrm{map}_{\C}(c, G(d))$$
is a weak equivalence for all $c \in \C$ and $d \in \D$. The natural transformation $u$ is the \emph{unit transformation} of the adjunction 
and it can be constructed using the bicartesian properties of the adjunction (see \cite[Proposition 5.2.2.8]{HTT}). If it exists, an adjoint is uniquely determined \cite[Proposition 5.2.6.2]{HTT}. As in ordinary category theory, left adjoints preserve colimits and right adjoints preserve limits \cite[Proposition 5.2.3.5]{HTT}.

\medskip

We will make use 
of the following useful criterion for recognizing adjoint functors which generalizes the classical description in terms of universal arrows.  

\begin{prop}\label{initialadjunction}
Let $q: \mathscr{M} \to \Delta^1$ be a cartesian fibration associated to a 
functor $G \colon \D \to \C$ where $\C= \mathscr{M}_0$ and $\D= \mathscr{M}_1$. Then the following are equivalent:
\begin{enumerate}
\item The functor $G$ has a left adjoint. 
\item The $\infty$-category $G_{c/}$ has an initial object for each $c \in \C$. 
\end{enumerate}
\end{prop}
\begin{proof}
This is a reformulation of \cite[Lemma 5.2.4.1]{HTT} using \cite[Proposition 4.4.4.5]{HTT} and \cite[Propositions 4.2.1.5 and 4.2.1.6]{HTT}. Another proof of this characterization can be found in \cite[Proposition 6.1.11]{Cis2}.
\end{proof}

\subsection{The main results} Freyd's classical \emph{General Adjoint Functor Theorem} states that a continuous functor $G \colon \D \to \C$ from a locally small and complete category is a right adjoint if and only if $G$ satisfies the solution set condition (see, for example, \cite[V.6, Theorem 2]{ML}, or \cite[Ch. 3, Exercise J]{Fre} for a little less general formulation). In general, the solution set condition 
is a weakening of the condition that the slice category $G_{c/}$ admits an initial object for each $c \in \C$. We consider this same condition for functors between 
$\infty$-categories.   

\begin{defn} \label{solution set}
Let $G \colon \D \to \C$ be a functor between $\infty$-categories. We say that $G$ satisfies the \textit{solution set condition} if the $\infty$-category $G_{c/}$ admits a small weakly initial set for any $c \in \C$.
\end{defn}  

As the following proposition shows, this solution set condition is again essentially $1$-categorical.

\begin{prop}\label{homotopyinitialset}
Let $G \colon \D \to \C$ be a functor between $\infty$-categories. Then $G$ satisfies the solution set condition if and only if $\h G \colon \h\D \to \h\C$ 
does.
\end{prop}
\begin{proof}
We will use the following property of the homotopy category which follows easily 
from its construction: for any $n=0,1,2$ and any lifting diagram of the form
\[
\begin{tikzcd}
\partial \Delta^n \rar \dar & \C \dar\\
\Delta^n \rar \urar[dashed] & \h\C
\end{tikzcd}
\]
there is a lift $\Delta^n \to \C$ which makes the diagram commute (cf. \cite[Remark 2.3.4.14]{HTT}). 

For $c\in \C$, the canonical functor $\C_{c/}\to (\h \C)_{c/}$ induces a functor $G_{c/} \to \h G_{c/}$. 
Suppose that $\h G$ satisfies the solution set condition, i.e. $\h G_{c/}$ admits a small weakly initial set $S_c$ for each $c \in \C$. Then given an object $(u \colon c \to G(d))\in G_{c/}$, there exists an object $(c \xrightarrow{[f_s]} G(d_s)) \in S_c$ and a morphism $([\varphi]\colon d_s \to d)\in \h\D$ such that the diagram
\[
\begin{tikzcd}
{} & c \drar{[u]} \dlar[swap]{[f_s]} & \\ 
G(d_s) \arrow{rr}[swap]{\h G[\varphi]} & & G(d)
\end{tikzcd}
\]
commutes in $\h\C$. Now choosing representatives $f_s$ and $\varphi$ for the classes $[f_s]$ and $[\varphi]$, we obtain a 2-boundary in $\C$ of the form
\[
\begin{tikzcd}
{} & c \drar{u} \dlar[swap]{f_s} & \\ 
G(d_s) \arrow{rr}[swap]{G(\varphi)} & & G(d).
\end{tikzcd}
\]
Using the property of the homotopy category mentioned above, there exists a 2-simplex in $\C$ with the given boundary. In particular, this 2-simplex determines a morphism $(f_s \to u) \in G_{c/}$. Thus, a set of representatives $f_s$, one for each $[f_s] \in S_c$, determines a small weakly initial set in $G_{c/}$. 

Conversely, it is easy to see that the image of any small weakly initial set in $G_{c/}$ under the functor $G_{c/} \to \h G_{c/}$ determines a small weakly initial set in $\h G_{c/}$.
\end{proof}

We also consider the following stronger condition which employs the notion of an $h$-initial object. This condition has no interesting analogue for ordinary categories. 

\begin{defn} \label{h-initial}
Let $G \colon \D \to \C$ be a functor between $\infty$-categories. We say that $G$ satisfies the \textit{h-initial object condition} if the $\infty$-category $G_{c/}$ admits an $h$-initial object for every $c \in \C$.
\end{defn}  

We may regard this condition as a stronger version of the solution set condition given that it asserts the existence of solution sets for both objects and 1-morphisms. 
 
\medskip 

Note that Freyd's General Adjoint Functor Theorem does not require any smallness assumptions on the target category. For the generalization of the theorem to $\infty$-categories, we need a new notion of smallness for $\infty$-categories. 
First recall that an $\infty$-category $\C$ is locally small if and only if for every  pair of objects $x, y \in \C$, the mapping space $\mathrm{map}_{\C}(x, y)$ is essentially small (see \cite[Proposition 5.4.1.7]{HTT}). Based on this characterization, we introduce the following more general notion. 

\begin{defn} 
Let $\C$ be an $\infty$-category. We say that $\C$ is $2$-\emph{locally small} if for every pair of objects $x, y \in \C$, the mapping space $\mathrm{map}_{\C}(x, y)$ is locally small.  
\end{defn} 
 
Note that every ordinary category (not necessarily locally small) is always $2$-locally small and every locally small $\infty$-category is also $2$-locally small. 

\medskip

We can now state our main adjoint functor theorems. The first one is a generalization of Freyd's General Adjoint Functor Theorem.  

\begin{thm}[GAFT]\label{GAFT}
Let $G \colon \D \to \C$ be a functor between $\infty$-categories. Suppose that $\D$ is locally small and complete and $\C$ is $2$-locally small. Then $G$ admits a left adjoint if and only if it is continuous and satisfies the solution set condition.
\end{thm}

Using instead the (stronger) $h$-initial object condition, we obtain our second adjoint functor theorem under weaker assumptions on the $\infty$-category $\D$ and no smallness assumption on $\mathscr C$. 

\begin{thm}[GAFT$_{\mathrm{fin}}$]\label{finGAFT}
Let $G \colon \D \to \C$ be a functor between $\infty$-categories, where $\D$ is finitely complete. Then $G$ admits a left adjoint if and only if it is finitely continuous and satisfies the $h$-initial object condition.
\end{thm}

For the proofs of these theorems, we will need the following lemmas. 

\begin{lem}\label{slicecomplete2}
Let $G: \D \to \C$ be a functor between $\infty$-categories and $c \in \C$. Suppose that $\D$ is (finitely) complete and $G$ is (finitely) continuous. Then $G_{c/}$ is (finitely) complete.
\end{lem}
\begin{proof}
This is shown similarly to \cite[Lemmas 5.4.5.2 and 5.4.5.5]{HTT} using that the functor $\C_{c/} \to \C$ preserves and reflects limits by \cite[Proposition 1.2.13.8]{HTT}.
\end{proof}

\begin{lem}\label{slicesmall}
Let $G: \D \to \C$ be a functor between $\infty$-categories,  where $\D$ is locally small and $\C$ is $2$-locally small. Then for every object $c \in \C$, the $\infty$-category $G_{c/}$ is locally small.
\end{lem}
\begin{proof}
We need to show that for every pair of objects $(u \colon c \to G(d)) \in G_{c/}$ and $(u' \colon c \to G(d')) \in G_{c/}$, the mapping space 
$$\mathrm{map}_{G_{c/}}(u, u')$$
is essentially small (see \cite[Proposition 5.4.1.7]{HTT}). The pullback square 
of $\infty$-categories
\[
\begin{tikzcd}
G_{c/} \arrow[d,"q"] \arrow[r] & \C_{c/} \arrow[d, "p"] \\
\D \arrow[r,"G"] & \C
\end{tikzcd}
\]
yields a homotopy pullback square of mapping spaces 
\begin{equation} \label{diagram-4}
\begin{tikzcd}
\mathrm{map}_{G_{c/}}(u, u') \arrow[d,"q"] \arrow[r] & \mathrm{map}_{\C_{c/}}(u, u') \arrow[d, "p"] \\
\mathrm{map}_{\D}(d, d') \arrow[r,"G"] & \mathrm{map}_{\C}(G(d), G(d')).
\end{tikzcd}
\end{equation}
Since $\C_{c/} \to \C$ is a left fibration, the (homotopy) fiber of the right vertical map is either empty or can be identified using \cite[Proposition 2.4.4.2]{HTT} with the mapping space
\begin{equation} \label{2-mapping-space}
\mathrm{map}_{p^{-1}(G(d'))}(u', u').
\end{equation}
Since $p^{-1}(G(d')) \simeq \mathrm{map}_{\C}(c, G(d'))$ is locally small by assumption, it follows that \eqref{2-mapping-space} is essentially small. Thus, the (homotopy) fibers of the left vertical map in \eqref{diagram-4} are essentially small. Then the result follows from \cite[Proposition 5.4.1.4]{HTT} since $\mathrm{map}_{\D}(d, d')$ is essentially small by assumption.
\end{proof}

\noindent \textbf{Proof of Theorem \ref{GAFT}.} Suppose that $G$ admits a left adjoint. Then using \cite[Proposition 5.2.3.5]{HTT}, the functor $G$ is continuous. By Proposition \ref{initialadjunction}, the $\infty$-cateogry $G_{c/}$ admits an initial object, which also defines a small weakly initial set.

Conversely, by Proposition \ref{initialadjunction}, it is enough to show that the $\infty$-category $G_{c/}$ admits an initial object for each $c \in \C$. By Lemma \ref{slicesmall}, $G_{c/}$ is locally small, and by Lemma \ref{slicecomplete2}, it is complete since $G$ is continuous. The $\infty$-category $G_{c/}$ admits a small weakly initial set by assumption. Therefore it also admits an initial object by Proposition \ref{initialcriterion}. 
\qed

\medskip

\noindent \textbf{Proof of Theorem \ref{finGAFT}.} Suppose that $G$ admits a left adjoint. Then using \cite[Proposition 5.2.3.5]{HTT}, the functor $G$ is (finitely) continuous. Moreover, for each $c \in \C$, the $\infty$-category $G_{c/}$ has an initial object by Proposition \ref{initialadjunction}, and therefore also an $h$-initial object.  

Conversely, suppose that $G_{c/}$ has an $h$-initial object for each $c \in \C$. 
By Lemma \ref{slicecomplete2}, the $\infty$-category $G_{c/}$ is finitely 
complete since $G$ is finitely continuous. Then Proposition \ref{finitecriterion} shows that $G_{c/}$ 
admits an initial object for each $c \in \C$, and therefore the result follows from the characterization in Proposition \ref{initialadjunction}.
\qed

\subsection{Adjunctions and homotopy categories} An adjunction $F \colon \C \rightleftarrows \D \colon G$ between $\infty$-categories induces an (ordinary) adjunction $\h F \colon \h \C \rightleftarrows \h\D \colon \h G$ between the homotopy categories. The converse statement, however, is false in general (for example, the canonical functor $\C \to \h\C$ 
does not admit a left or a right adjoint in general). 

In this subsection we show the converse statement under appropriate additional assumptions. Our main results are direct consequences of the general adjoint functor theorems of the previous subsection. They will essentially follow from the observation that both the solution set condition of Definition \ref{solution set} and the $h$-initial object condition of Definition \ref{h-initial} can be tested at the level of the homotopy category.

\begin{thm}\label{finitehomotopyadjunction}
Let $\D$ be an $\infty$-category which admits finite limits and $\C$ an $\infty$-category. Let $G \colon \D \to \C$ be a functor which preserves finite limits. 
Then $G$ admits a left adjoint if and only if $\h G \colon \h \D \to \h \C$	does.
\end{thm}

The proof requires the following technical lemma.

\begin{lem}\label{initialcomparison}
Let $G \colon \D \to \C$ be a functor between (ordinary) categories. Suppose that $G$ is surjective on objects, full, and conservative. Furthermore, suppose that for any pair of morphisms $f,g\colon d\to d'$ in $\D$, there exists a morphism $u_{f,g}\colon w \to d$ such that $f\circ u_{f,g}=g\circ u_{f,g}$. Then $x \in \D$ is  initial if and only if $G(x)$ is initial in $\C$.
\end{lem}
\begin{proof}
Since $G$ is full and surjective on objects, it follows that it preserves initial objects. Conversely, suppose that $G(x)$ is initial in $\C$ for some object $x \in \D$. We claim that $x$ is initial in $\D$. It is clear that $x$ is weakly initial, since $G$ is full. Suppose we have two morphisms $f,g \colon x \to d$ in $\D$. By assumption, there exists a morphism $u_{f,g} \colon w \to x$ which equalizes $f$ and $g$. The induced morphism $G(w)\to G(x)$ admits a section $s \colon G(x) \to G(w)$, since $G(x)$ is initial. Using that $G$ is full, we find a morphism $v\colon x\to w$ such that $G(v)=s$. Since $G$ is conservative, the composition $u_{f,g} \circ v \colon x \to x$ is an isomorphism. This means that $u_{f,g}$ is a (split) epimorphism which implies that $f = g$. 
\end{proof}

\noindent \textbf{Proof of Theorem \ref{finitehomotopyadjunction}.}
By Theorem \ref{finGAFT}, $G$ has a left adjoint if and only if $\h(G_{c/})$ has an initial object for each $c \in \C$. On the other hand, $\h G$ has a left adjoint if and only if the slice category $\h G_{c/}$ has an initial object. Thus, we will need to compare initial objects of $\h(G_{c/})$ with initial objects of $\h G_{c/}$.

There is a canonical functor $\C_{c/}\to (\h \C)_{c/}$ which is given by sending morphisms to their homotopy classes. This is surjective on objects and full by 
the construction of the homotopy category. It induces a functor between the slice $\infty$-categories $G_{c/} \to \h G_{c/}$ which descends to a functor from the homotopy category
$$\h(G_{c/})\to \h G_{c/}.$$
Then it suffices to show that this last functor satisfies the assumptions of Lemma \ref{initialcomparison}.
It is easy to see that it is surjective on objects and full (see 
the proof of Proposition \ref{homotopyinitialset}). It is also conservative 
because equivalences in $G_{c/}$ are detected in $\D$. Therefore it remains to show that for each pair of morphisms $[\varphi_1],[\varphi_2]\colon f \to g$ in $\h(G_{c/})$, there exists a morphism that equalizes them. Choose representatives $\varphi_1$ and $\varphi_2$ in $G_{c/} $. Since $\D$ has finite limits and $G$ preserves them, $G_{c/}$ also has finite limits by Lemma \ref{slicecomplete2}. In particular, there is an equalizer of $\varphi_1$ and $\varphi_2$,
\[
\begin{tikzcd}
w \rar{\psi}& f \rar[shift left = 1.3]{\varphi_1} \rar[shift right =1.3,swap]{\varphi_2} &g.
\end{tikzcd}
\]
Then $[\psi]$ equalizes $[\varphi_1]$ and $[\varphi_2]$ and the result follows.
\qed

\begin{remark} \label{homotopy-adjunction-rem}
A weaker version of Theorem \ref{finitehomotopyadjunction} can also be obtained directly from Theorem \ref{GAFT}. Let $\D$ be a locally small and complete $\infty$-category and let $G \colon \D \to \C$ be a continuous functor, where $\C$ is 
$2$-locally small. Then $G$ admits a left adjoint if and only if $\h G$ does -- because $G$ satisfies the solution set condition if $\h G$ does (Proposition \ref{homotopyinitialset}).
\end{remark}

\begin{example}
We mention a few elementary examples to indicate the necessity of the assumptions in Theorem \ref{finitehomotopyadjunction}. First, it is crucial that the functor $G$ preserves finite limits. For example, given a finitely complete $\infty$-category $\D$, the canonical functor $\gamma \colon \D \to \h \D$ is not an adjoint in general, even though it clearly induces one between 
the homotopy categories. Moreover, this assumption in Theorem \ref{finitehomotopyadjunction} is also needed even in the case where a functor $F \colon \C \to \D$ exists such that the induced pair of functors $\h F \colon \h \C \rightleftarrows \h \D \colon \h G$ forms an adjoint pair. For example, let $\C = N_{\Delta}(\widehat{\C})$ be the coherent nerve of the simplicial category $\widehat{\C}$ which has exactly two objects $x$ and $y$ and only non-identity morphisms from $x$ to $y$. There is a pair of functors $F \colon \C \rightleftarrows \Delta^0 \colon G$, where 
$G$ is defined by $y \in \C$. Then $(\h F, \h G)$ is an adjoint pair if and only if $\mathrm{map}_{\C}(x, y)$ is non-empty and connected, whereas $(F, G)$ is an adjoint pair if and only if 
the mapping space is contractible -- in which case, $G$ indeed preserves finite limits. In addition, it is also worth remarking that under the assumptions of Theorem \ref{finitehomotopyadjunction}, given a functor $F \colon \C \to \D$ such that $(\h F, \h G)$ is an adjoint pair, it does not follow that $F$ is an adjoint of $G$, as there may be different 
such functors which induce the same functor on the homotopy category. For example, for $\C = N_{\Delta}(\widehat{\C})$ as before, where $\mathrm{map}_{\C}(x, y)$ is connected but 
not contractible, the functor $F$ defined as the composition $\C \to \h \C \simeq \Delta^1 \xrightarrow{u} \C$, for some $u \colon x \to y$ in $\C$, induces the identity functor between the homotopy categories, but $F$ is clearly not an adjoint of the identity. 
\end{example}

An interesting special case of Theorem \ref{finitehomotopyadjunction} is the following result about equivalences of $\infty$-categories. This result is shown  using different methods in \cite[Theorem 7.6.10]{Cis2} and a weaker version of the result can also be found in \cite[Proposition 2.15]{Bar}. Analogous results for Waldhausen categories are obtained in \cite{BM} and \cite{Cis}. 

\begin{coro}
Let $\C$, $\D$ and $G \colon \D \to \C$ be as in Theorem \ref{finitehomotopyadjunction}. Then $G$ is an equivalence of $\infty$-categories 
if and only if $\h G$ is an equivalence of (ordinary) categories.	
\end{coro}
\begin{proof}
Note that for any $\infty$-category $\C$ the canonical functor
\[
\mathrm{Fun}(\C,\C)\to \mathrm{Fun}(\C,\h\C)\simeq \mathrm{Fun}(\h\C,\h\C)
\]
is conservative. Hence by Theorem \ref{finitehomotopyadjunction}, if $\h G$ is an equivalence then $G$ admits a left adjoint $F \colon \C \to \D$ such that the unit and counit transformations of the adjunction $(F, G)$ are natural equivalences of functors. 
The converse is obvious. 
\end{proof} 

\begin{example}
Let $\D$ be an $\infty$-category which admits finite limits. Suppose that the canonical functor $\D \to \h\D$ preserves finite limits. Then $\D \simeq \h\D$, i.e., $\D$ is equivalent to a finitely complete ordinary category.
\end{example}

\section{Adjoint functor theorems for presentable $\infty$-categories}

\subsection{Statement of results} Adjoint functor theorems for presentable $\infty$-categories are obtained in \cite{HTT} where the theory of presentable $\infty$-categories is developed extensively. These theorems give useful characterizations for a functor between presentable $\infty$-categories to be a left or a right adjoint and they generalize analogous classical results about locally presentable categories.  

In this section we recover these adjoint functor theorems as applications 
of GAFT (Theorem \ref{GAFT}). The first one identifies right adjoint functors 
and recovers \cite[Corollary 5.5.2.9(2)]{HTT}. 

\begin{thm}[RAFT] \label{RAFT}
Let $\D$ be a presentable $\infty$-category and $\C$ a locally small $\infty$-category such that each object $c \in \C$ is $\kappa$-compact for some regular cardinal $\kappa$.
Let $G \colon \D \to \C$ be a functor. 
\begin{enumerate}
\item[(1) ] If $G$ preserves small limits and is accessible (i.e. it preserves $\kappa$-filtered colimits for some regular cardinal $\kappa$), then $G$ is a right adjoint.
\item[(2)] Suppose that $\C$ is an accessible $\infty$-category. If $G$ is a right adjoint, then it is accessible and preserves all small limits. 
\end{enumerate} 
\end{thm}

The corresponding result of \cite{HTT} for left adjoint functors states that a functor between presentable $\infty$-categories is a left adjoint if and only if 
it preserves small colimits \cite[Corollary 5.5.2.9(1)]{HTT}. We prove a version  of this for more general $\infty$-categories that satisfy the following property.

\begin{defn}
We say that an $\infty$-category $\C$ has an \emph{essentially small colimit-dense subcategory}
if there is an essentially small full subcategory $\C_0 \subset \C$ such that every object $x \in \C$ is a colimit of a diagram $K \to \C_0 \subset \C$ with values in $\C_0$, where $K$ is small simplicial set.
\end{defn}

\begin{thm} \label{LAFT}
Let $\C$ be a locally small cocomplete $\infty$-category and $\D$ a locally small $\infty$-category. Suppose that $\C$ has an essentially small colimit-dense subcategory. Then 
a functor 
$F \colon \C \to \D$ is a left adjoint if and only if it preserves small colimits.
\end{thm}

Both Theorem \ref{RAFT} and Theorem \ref{LAFT} are consequences of GAFT, but the proofs in each case are quite different. We point out that we do not know an $\infty$-categorical generalization of Freyd's Special Adjoint Functor Theorem (SAFT) -- see, for example, \cite[V.8, Theorem 2]{ML}. However, in view of the assumptions for Theorem \ref{LAFT}, one may regard this as a special 
$\infty$-categorical SAFT. 

Presentable $\infty$-categories, and more generally, localizations of presheaf $\infty$-catego-\-ries, are examples of $\infty$-categories which admit essentially small colimit-dense subcategories. Thus, we obtain the following result which recovers \cite[Corollary 5.5.2.9(1)]{HTT}.

\begin{coro}[LAFT] \label{LAFT2}
Let $\C$ be a presentable $\infty$-category and $\D$ a locally small $\infty$-category. Then a functor 
$F \colon \C \to \D$ is a left adjoint if and only if it preserves small colimits.
\end{coro}

It is well known that presentable $\infty$-categories are complete \cite[Corollary 5.5.2.4]{HTT}. The following corollary of Theorem \ref{LAFT} shows existence of limits under more general assumptions. This generalizes to $\infty$-categories an analogous result for ordinary categories shown in \cite{AHR}.

\begin{coro}
Let $\C$ be a locally small cocomplete $\infty$-category which has an essentially small colimit-dense subcategory. Then $\C$ is complete.
\end{coro}
\begin{proof}
Let $K$ be a small simplicial set and consider the constant $K$-diagram functor
$c\colon \C \to \C^K.$	
Since colimits in $\C^K$ are computed pointwise \cite[Corollary 5.1.2.3]{HTT}, \cite[Corollary 6.2.10]{Cis2}, the constant functor preserves small colimits. Moreover, $\C^K$ is again locally small (see \cite[Example 5.4.1.8]{HTT}). Thus, by Theorem \ref{LAFT}, the constant functor admits a right adjoint. By \cite[Proposition 6.2.9]{Cis2} or \cite[Lemma 4.2.4.3]{HTT}, this implies that $\C$ has limits indexed by $K$.
\end{proof}

\subsection{Proof of Theorem \ref{RAFT}} We first prove (1). Suppose that $G$ is accessible and preserves small limits. Note that the presentable $\infty$-category $\D$ is also complete by \cite[Corollary 5.5.2.4]{HTT}. By Theorem \ref{GAFT}, it suffices to show that for every object $c \in \C$, the $\infty$-category $G_{c/}$ has a small weakly initial set. By assumption, we may suppose that the object $c \in \C$ is $\kappa$-compact 
where $\kappa$ is large enough so that $\D$ is $\kappa$-presentable and $G$ preserves $\kappa$-filtered colimits. The full subcategory $\D^{\kappa} \subset \D$
of $\kappa$-compact objects is essentially small so that there is an equivalent small subcategory $\widetilde{\D} \simeq \D^{\kappa} \subset \D$. Since $\C$ is locally small, we may consider a small set of objects,  
$$S = \{c \to G(d) \ | \  d \in \widetilde{\D}\} \subset G_{c/},$$ 
which consists of one morphism $(c \to G(d))$ from each homotopy class. We claim that this defines a weakly initial set in $G_{c/}$. To see this, consider an object $u \colon c \to G(x)$ 
in $G_{c/}$ where $x \in \D$. 
Since $\D$ is $\kappa$-presentable, there is a $\kappa$-filtered diagram of $\kappa$-compact objects $\Phi \colon K \to \widetilde{\D} \hookrightarrow \D$ whose colimit is $x \in \D$. In addition, since the functor $G$ preserves $\kappa$-filtered colimits, $\colim_K (G \circ \Phi)$ exists in $\C$ and there is an equivalence $G(x) \simeq \colim_{K} (G \circ \Phi)$. Moreover, since $c$ is $\kappa$-compact in $\C$, we obtain a factorization in $\C$ as follows 
\[
\begin{tikzcd}[column sep=small]
& c \arrow[dr, "u"] \arrow{dl}[above]{v} \\
G(\Phi(i)) \arrow[rr,"G(f_i)"] && G(x),
\end{tikzcd}
\]
where $f_i \colon \Phi(i) \to x$ is the canonical morphism to the colimit of $\Phi$ and $v \in S$. This shows the claim, and therefore $G$ is a right adjoint by Theorem \ref{GAFT}.

For (2), suppose that $G$ is a right adjoint. Then it preserves limits by \cite[Proposition 5.2.3.5]{HTT}. Let $F$ be a left adjoint of $G$ and suppose that $\C$ is 
$\kappa$-accessible. We may choose a regular cardinal $\tau \geq \kappa$ such that for each $\kappa$-compact object $c \in \C$, $F(c) \in \D$ is $\tau$-compact. 

We claim that $G$ preserves $\tau$-filtered colimits. To see this, consider a $\tau$-filtered diagram $\Phi \colon K \to \D$ and the canonical morphism in $\C$
\begin{equation} \label{comparison-map}
\mathrm{colim}_K (G \circ \Phi) \longrightarrow G(\mathrm{colim}_K \Phi).
\end{equation}
For a $\kappa$-compact object $c \in \C$, we consider the following homotopy commutative diagram of mapping spaces:
\[
\begin{tikzcd}
\mathrm{map}_{\C}(c, \mathrm{colim}_K (G \circ \Phi)) \arrow[r] & \mathrm{map}_{\C}(c, G(\mathrm{colim}_K \Phi))  \\
& \mathrm{map}_{\D}(F(c), \mathrm{colim}_K \Phi) \arrow[u, "\simeq"] \\
\mathrm{colim}_K \mathrm{map}_{\C}(c, G(\Phi(-))) \arrow[r, "\simeq"] \arrow[uu, "\simeq"] & 
\mathrm{colim}_K \mathrm{map}_{\D}(F(c), \Phi(-)) \arrow[u, "\simeq"] 
\end{tikzcd}
\]
where the indicated equivalences follow either from the adjunction or from the $\tau$-compactness of $c$ and $F(c)$. Then the top map is also a weak equivalence for each $\kappa$-compact $c \in \C$. Since $\C$ is $\kappa$-accessible, it follows that \eqref{comparison-map} is an equivalence, as required. 
\qed

\subsection{Proof of Theorem \ref{LAFT}}

A left adjoint preserves colimits by \cite[Proposition 5.2.3.5]{HTT}. Conversely, suppose that $F \colon \C \to \D$ preserves small colimits. After passing to the opposite $\infty$-categories, it suffices to show that the continuous functor $F^{\op} \colon \C^{\op} \to \D^{\op}$ admits a left adjoint. Since $\C^{\op}$ is complete and $F^{\op}$ is continuous, it suffices by Theorem \ref{GAFT} to show that $F^\op$ satisfies the solution set condition. Using the canonical isomorphism of simplicial sets
$$F_{d/}^{\op} \cong (F_{/d})^{\op},$$
this is equivalent to showing that $F_{/d}$ admits a weakly terminal set. 

Let $\iota:\C_0 \subset \C$ be an essentially small colimit-dense subcategory of $\C$. Consider the following diagram of pullback  squares
\begin{equation}\label{diagramLAFT}
\begin{tikzcd}
(F \circ \iota)_{/d} \arrow[d] \arrow[r, "j"] & F_{/d} \arrow[d] \arrow[r] & \D_{/d} \arrow[d] \\
\C_0 \arrow[r,hookrightarrow,"\iota"] & \C \arrow[r,"F"] & \D.
\end{tikzcd}
\end{equation}
Since $\C$ is cocomplete and $F$ preserves all small colimits, the $\infty$-category $F_{/d}$ is also cocomplete -- by the dual of Lemma \ref{slicecomplete2}. 
Moreover, since $\D$ is locally small and $\C_0$ is essentially small, the $\infty$-category $(F \circ \iota)_{/d}$ is essentially small (see \cite[Proposition 5.4.1.4]{HTT}).
In particular, we can form the colimit of the diagram $j \colon (F \circ \iota)_{/d} \to F_{/d}$. We denote the colimit-object by $(e,f:F(e) \to d) \in F_{/d}$ and claim that this is weakly terminal in $F_{/d}$. To see this, consider an object $(c,g:F(c) \to d) \in F_{/d}$. 
We can write $c \in \C$ as a colimit of a diagram 
\[
\begin{tikzcd}[column sep=small]
K \arrow[r,"\phi"] & \C_0 \arrow[r,hookrightarrow,"\iota"] & \C
\end{tikzcd}
\]
where $K$ is a small simplicial set. The functor $F$ is cocontinuous, so the colimit cocone $\overline \phi$ in $\C$ maps to a colimit cocone in $\D$ with colimit-object $F(c)$,
\[
\begin{tikzcd}[column sep=small]
K^{\triangleright} \arrow[r,"\bar{\phi}"] & \C \arrow[r,"F"] & \D.
\end{tikzcd}
\]
We can extend the morphism $g:F(c) \to d$ in $\D$ to a morphism of cocones by choosing an extension in the following diagram
\[
\begin{tikzcd}
K \star \Delta^0 \cup _{\emptyset \star \Delta^0} \emptyset \star \Delta^1 \arrow[rr, "F \circ \bar{\phi} \cup g"] \arrow[d] && \D \\
K \star \Delta^1 \arrow[urr,dashed,bend right, "\psi"]
\end{tikzcd}
\]
which exists since the vertical map is inner anodyne. 
Therefore we obtain a commutative diagram of the form
\[
\begin{tikzcd}[column sep=small]
&&& \D_{/d} \arrow[d] \\
K \arrow[urrr,bend left, "\psi_{| K \star \Delta^{\{1\}}}"] \arrow[r,"\phi"] & \C_0 \arrow[r,hookrightarrow,"\iota"] & \C \arrow[r,"F"] & \D,
\end{tikzcd}
\]
which combined with \eqref{diagramLAFT} extends to a commutative diagram as follows,
\[
\begin{tikzcd}[column sep=small]
&(F \circ \iota)_{/d} \arrow[d] \arrow[r, "j"] & F_{/d} \arrow[d] \arrow[r] & \D_{/d} \arrow[d] \\
K \arrow[r,"\phi"] \arrow[ur,"\phi'"] & \C_0 \arrow[r,hookrightarrow,"\iota"] & \C \arrow[r,"F"] & \D.
\end{tikzcd}
\]
Similarly, using the adjoint of $\psi$, we obtain an extension of the composite diagram $j \circ \phi'$ to a cocone $K^{\triangleright} \to F_{/d}$ with cone-object $(c, g \colon F(c) \to d)$. 
By \cite[Lemma 5.4.5.5]{HTT}, this is a colimit cocone since it is defined by the colimit cocone $\bar{\phi}$ in $\C$, the induced colimit cocone $F \circ \bar{\phi}$ in $\D$, and the cocone in $\D_{/d}$ induced by $\psi$ which is also a colimit cocone by \cite[Proposition 1.2.13.8]{HTT}. Therefore we obtain a canonical morphism between colimit-objects in 
$F_{/d}$,
$$(c,g) \simeq \mathrm{colim}_K (j\circ \phi') \longrightarrow 
\mathrm{colim}_{(F \circ \iota)_{/d}} j \simeq (e, f).$$
This shows that $(e, f)$ is weakly terminal in $F_{/d}$, as required. 
\qed

\begin{remark}
The weakly terminal object $(e, f) \in F_{/d}$ that is defined in the proof of Theorem \ref{LAFT} can actually be shown to be terminal in $F_{/d}$ when $\C$ is $\kappa$-presentable and $\C_0$ is the full subcategory of $\kappa$-compact objects. Hence it determines the value of the right adjoint at $d \in \D$. 
\end{remark}

\section{Brown representability for $\infty$-categories}

\subsection{Preliminaries} Let $\C$ be a locally small $\infty$-category. We write $[x, y]$ to denote the (small) set of morphisms from $x$ to $y$ in the homotopy category $\h \C$. A functor $F: \C^{op} \to \Set \ (:= \Set_{\mathbb{U}})$ is \emph{representable} if it is naturally  
isomorphic to a functor $[-, x] : \C^{op} \to \Set$ for some object $x \in \C$. Every representable functor satisfies the conditions (B1)--(B2) below.

\begin{defn} \label{Brown-rep-prop}
Let $\C$ be a locally small cocomplete $\infty$-category. We say that $\C$ \emph{satisfies Brown representability} if for any given functor $F \colon \C^{\op} \to \Set$, $F$ is representable  if (and only if) the following conditions are satisfied. 
\begin{enumerate}
\item[(B1).] For any small coproduct $\coprod_{i\in I} x_i$ in $\C$, the canonical morphism
$$
F\left(\coprod_{i\in I} x_i\right)\longrightarrow \prod_{i\in I} F(x_i)
$$
is an isomorphism.
\item[(B2).] For every pushout diagram in $\C$
\begin{equation*} 
\xymatrix{
x \ar[r] \ar[d] & y \ar[d] \\ 
z \ar[r] & w}
\end{equation*}
the canonical morphism
$$
F(w)\longrightarrow F(y)\times_{F(x)}F(z)
$$
is an epimorphism.
\end{enumerate}
\end{defn}

\begin{remark}
If $I$ is the empty set, then property (B1) says that $F$ sends the initial object of $\C$ to a set with one element. Note that we have not assumed that $\C$ is pointed. If this is the case, then every functor $F$ satisfying (B1) is canonically pointed, i.e. it factors through the category of pointed sets 
$\Set_*$.
\end{remark}

The purpose of a Brown representablity theorem for (cocomplete) $\infty$-categories is to identify a class of $\infty$-categories that satisfy Brown representability. This property is intimately connected with adjoint functor theorems. Recall that a \emph{weak (co)limit} in a category is a (co)cone on a diagram which satisfies the existence but not the uniqueness property of a (co)limit (co)cone. We then have the following proposition. 

\begin{prop} \label{Brown-implies-adjoints}
Let $\C$ and $\D$ be locally small $\infty$-categories and let $F \colon \C \to \D$ be a functor. Suppose that $\C$ is cocomplete and that it satisfies 
Brown representability. Then the induced functor $\h F \colon \h \C \to \h \D$ admits a right adjoint if and only if $F$ satisfies the following properties.
\begin{itemize}
\item[(B1$'$).] $F$ sends small coproducts in $\C$ to coproducts in $\h \D$.  
\item[(B2$'$).] $F$ sends pushout squares in $\C$
\begin{equation*}
\xymatrix{
x \ar[r] \ar[d] & y \ar[d] \\ 
z \ar[r] & w}
\end{equation*}
to weak pushout squares in $\h \D$
$$
\xymatrix{
F(x) \ar[r] \ar[d] & F(z) \ar[d] \\ 
F(y) \ar[r] & F(w).
}
$$ 
\end{itemize}
\end{prop}

\begin{proof}
The functor $\h F$ admits a right adjoint if and only if for every $d \in \D$ the functor 
$$Y_d \colon \h \C^{\op} \to \Set, \   \ c \mapsto [F(c), d]$$
is representable. Suppose that $F$ satisfies (B1$'$)-(B2$'$). Since $F$ satisfies (B1$'$), it follows that the functor 
$\C^{\op} \to \Set$, $c \mapsto [F(c), d]$, satisfies (B1). Moreover, it satisfies (B2) because $F$ satisfies (B2$'$). Since $\C$ satisfies Brown representability by assumption, 
it follows that $Y_d$ is representable for any $d \in \D$, and therefore 
$\h F$ admits a right adjoint.

Conversely, suppose that $\h F$ is a left adjoint. Then it preserves coproducts and weak pushouts. Coproducts in $\C$ define coproducts in $\h\C$, hence (B1$'$) is satisfied. 
A pushout diagram in $\C$
\begin{equation*}
\xymatrix{
x \ar[r] \ar[d] & y \ar[d] \\ 
z \ar[r] & w}
\end{equation*} 
defines a weak pushout in $\h\C$. To see this, note that for any $c \in \C$, we have
$$[w, c] \cong \pi_0(\mathrm{map}_{\C}(y, c) \times_{\mathrm{map}_{\C}(x, c)} \mathrm{map}_{\C}(z, c)),$$
where the pullback is formed in the $\infty$-category of spaces, and therefore the canonical map
$$[w, c] \to [y, c] \times_{[x, c]} [z, c]$$
is surjective. Hence (B2$'$) is also satisfied. 
\end{proof}

\begin{coro} \label{Brown-implies-complete} 
Let $\C$ be a locally small cocomplete $\infty$-category which satisfies Brown representability. Then $\C$ is complete.
\end{coro}
\begin{proof}
Let $K$ be a small simplicial set and let $c \colon \C \to \C^K$ denote the constant $K$-diagram functor. We need to show that $c$ admits a right adjoint for any $K$. Since colimits in $\C^K$ are computed pointwise, it follows that $c$ preserves small colimits (see \cite[Corollary 5.1.2.3]{HTT} or \cite[Corollary 6.2.10]{Cis2}). As a consequence, the functor $c$ satisfies (B1$'$)--(B2$'$) of Proposition \ref{Brown-implies-adjoints}. Moreover, $\C^K$ is again locally small (see \cite[Example 5.4.1.8]{HTT}). Since $\C$ satisfies Brown representability, it follows that the induced functor $\h(c) \colon \h \C \to \h(\C^K)$ admits a right adjoint. By applying Theorem \ref{finitehomotopyadjunction} or Remark \ref{homotopy-adjunction-rem}, we conclude that the functor $c \colon \C \to \C^K$ admits a right adjoint. 
\end{proof}

Similarly to the proof of Corollary \ref{Brown-implies-complete}, we may more generally combine Proposition \ref{Brown-implies-adjoints} and Theorem \ref{finitehomotopyadjunction} (or Remark \ref{homotopy-adjunction-rem}) to obtain the following:

\begin{coro}
Let $\C$ and $\D$ be locally small $\infty$-categories. Suppose that $\C$ is cocomplete and satisfies Brown representability. Then a functor $F \colon \C \to \D$ admits a right adjoint if and only if it preserves small colimits.
\end{coro}

\begin{remark}
A functor $F: \C^{\op} \to \Set$ is representable if the induced functor $\h F^{\op} \colon \h \C \to \Set^{\op}$ is a left adjoint -- the converse also holds if $\C$ satisfies Brown representability. Thus, the Brown representability property of Definition \ref{Brown-rep-prop} is a special case of Proposition \ref{Brown-implies-adjoints}.
\end{remark}

\begin{example} \label{ordinary-cats-Brown}
Suppose that $\C$ is an ordinary (locally small) cocomplete category which satisfies Brown representability. Let $F \colon \C \to \D$ be a functor which satisfies (B1$'$)-(B2$'$), where $\D$ is an ordinary category. Then $F$ admits a right adjoint by Proposition \ref{Brown-implies-adjoints}. In particular, $F$ actually preserves all colimits.
\end{example}

\begin{example} \label{comparison}
This example is an additional observation on the comparison between Theorem \ref{finitehomotopyadjunction} and Proposition \ref{Brown-implies-adjoints}. Let $\C$ be an $\infty$-category which admits finite colimits. The canonical functor $\gamma \colon \C \to \h \C$ satisfies (B1$'$)-(B2$'$) and $\h(\gamma)$ is obviously an equivalence, but $\gamma$ does not admit a right adjoint in general -- since it does not preserve pushouts in general.  
\end{example}

Examples of $\infty$-categories which satisfy Brown representability are discussed in the following subsections. We note first the following elementary proposition that will allow us to generate new examples from old ones. 

\begin{prop} \label{local-pres}
Let $\C$ be a locally small cocomplete $\infty$-category that satisfies Brown representability. 
If $\D$ is a localization of $\C$, then $\D$ also satisfies Brown representability.
\end{prop}
\begin{proof}
We may assume that $i \colon \D \subset \C$ is a full subcategory. Let $L \colon \C \to \D$ denote the left adjoint to the inclusion. Given a functor $F \colon \D^{op} \to \Set$ that satisfies (B1)-(B2), the composite functor $\C^{op} \to \D^{op} \to \Set$ also satisfies 
(B1)-(B2). Since $\C$ satisfies Brown representability by assumption, the functor $F \circ L^{op}$ is representable by an object $x \in \C$. Moreover, this composite functor sends $L$-equivalences to isomorphisms since it factors through $\D^{\op}$. This implies that $x$ is $L$-local, that is, $x \simeq i L(x)$. Then $F$ is representable by the object $L(x)$.   
\end{proof}

\subsection{Compactly generated $\infty$-categories} Various generalizations of Brown's original representability theorem \cite{Br} have been obtained over the years in 
different contexts. In this subsection, we revisit Heller's formulation \cite{He} of Brown's original argument as it applies in the context of compactly generated $\infty$-categories.   

\begin{defn} \label{weak_gen_def}
Let $\C$ be a locally small $\infty$-category. A set of \emph{weak generators} of $\C$ is a small set of objects $\mathcal{G}$ that jointly detect equivalences, i.e. a morphism $f \colon x\to y$ in $\C$ is an equivalence if and only if the 
canonical morphism
$$
\xymatrix{[g,x] \ar[r]^{f_*} & [g,y]}
$$
is an isomorphism for every object $g \in \mathcal{G}$.
\end{defn}

Note that the existence of weak generators in $\C$ depends only on the homotopy category 
$\h \C$. This property should not be confused with the strictly stronger property which refers to a set of objects that jointly distinguish parallel arrows. A related notion  refers to a set of objects that detects whether the canonical morphism to the terminal object $x \to *$ is an equivalence. For stable $\infty$-categories, this is equivalent to a set of weak generators, but the two concepts are not equivalent in general. For example, the collection of spheres $\{S^n\}_{n \geq 0}$ is not a set of weak generators in the $\infty$-category of spaces.

\begin{defn} \label{h-compact}
Let $\C$ be a locally small $\infty$-category which admits $\omega$-indexed colimits. An object $x \in \C$ is \emph{h-compact} if for every diagram 
$F\colon \omega \to \C$ the canonical morphism
$$
\colim_{i <\omega}[x, F(i)]\to [x,\colim_{i <\omega} F(i)]
$$
is an isomorphism.
\end{defn}

This is a familiar concept and one of the many thematic variations that appear in the literature. The main point to notice about the definition is that it is given in terms of the homotopy category but with respect to colimit diagrams in $\C$ - 
rather than weak colimit diagrams in $\h \C$, or strict colimit diagrams in some 
category with weak equivalences which models $\C$.  

\begin{defn} \label{def-comp-gen}
A locally small $\infty$-category $\C$ is called \emph{compactly generated} if it admits small colimits and has a set of weak generators $\mathcal{G}$ consisting of $h$-compact objects. 
\end{defn}

\begin{example} \label{loc-pres-cat} Let $\C$ be an ordinary (locally small) category which admits small colimits. Then $\C$ is compactly generated if $\C$ has a set of finitely presentable objects which jointly detect isomorphisms in $\C$, that is, if $\C$ is locally finitely presentable.   
\end{example}

\begin{example} \label{example-spectra}
The stable $\infty$-category of spectra $\mathcal{S}p$ is compactly generated. Moreover, for every small simplicial set $K$, the $\infty$-category $\mathcal{S}p^K$ is also compactly generated. More generally, any  stable finitely presentable $\infty$-category is compactly generated. 
\end{example}

\begin{prop} \label{intermediate}
Let $\C$ be a compactly generated $\infty$-category. 
\begin{enumerate}
\item The homotopy category $\h\C$ has small coproducts and weak pushouts. 
\item Every diagram $F \colon \omega \to \h \C$ admits a weak colimit $\colim^w F \in \h \C$ such that the canonical morphism 
$$\colim_{i < \omega} [x, F(i)] \to [x, \colim^w F]$$
is an isomorphism for each $h$-compact object $x \in \C$. 
\end{enumerate}
\end{prop}
\begin{proof}
As explained in the proof of Proposition \ref{Brown-implies-adjoints}, $\h\C$ inherits coproducts and weak pushouts from the coproducts and the pushouts 
in $\C$. This shows (1). For (2), we may lift the diagram $F$ to 
a diagram $\widetilde{F} \colon \omega \to \C$ and let $\colim^w F$ be the 
colimit of $\widetilde{F}$, regarded as an object in $\h \C$. This colimit can be obtained by a telescope construction which corresponds to a pushout diagram in $\C$ as follows 
$$
\xymatrix{
\coprod_i (F(i) \coprod F(i)) \ar[rr]^(0.6){(\mathrm{id}, F)} \ar[d]_{\nabla} && \coprod_i F(i) \ar[d] \\
\coprod_i F(i) \ar[rr] && \colim \widetilde{F}.
}
$$
It follows that $\colim^w F$ defines a weak colimit of $F$ in $\h \C$. Moreover, it has the required property with respect to every $h$-compact object by definition. 
\end{proof}

Using Proposition \ref{intermediate}, Heller's theorem \cite[Theorem 1.3]{He} yields the following result for compactly generated $\infty$-categories. The 
result is a small generalization of the Brown representability theorem in \cite[Theorem 1.4.1.2]{HA}. 

\begin{thm} \label{Heller}
Every compactly generated $\infty$-category satisfies Brown representability.
\end{thm}
\begin{proof} 
Let $F \colon \C^{\op} \to \Set$ be a functor which satisfies (B1)-(B2), where $\C$ is a compactly generated $\infty$-category. (B1) implies that the induced 
functor $\h F \colon \h \C^{\op} \to \Set$ sends coproducts to products, since these agree with the coproducts in $\C$. Moreover, $\h F$ sends every weak pushout in $\h \C$ to a weak pullback, since this holds by (B2) for the choices of weak pushouts that arise from pushouts in $\C$. Then, using Proposition \ref{intermediate}, the representability of $\h F$  is an application of \cite[Theorem 1.3]{He}. (We note that the dual of \emph{right cardinally bounded} in \cite{He} is stronger than $h$-compactness in that it requires that the property of Definition \ref{h-compact} holds for any diagram indexed by a sufficiently large regular cardinal. However, the proof in \cite{He} only requires that a single such regular cardinal exists.)
\end{proof}

\begin{example} \label{loc-pres-cat-2}
By Theorem \ref{Heller} and Example \ref{loc-pres-cat}, every locally finitely presentable category satisfies Brown representability. Moreover, every locally presentable category satisfies Brown representability. This follows either from Proposition 
\ref{local-pres} since every locally presentable category is equivalent to a full reflective subcategory of a locally finitely presentable category or by applying 
\cite[Theorem 1.3]{He} directly. In particular, using Remark \ref{ordinary-cats-Brown}, we conclude that given a functor $F \colon \C \to \D$ between ordinary (locally small) categories, where $\C$ is locally presentable, $F$ is a left adjoint if and only if $F$ preserves small coproducts and weak pushouts.
\end{example}

\begin{example} \label{stable-fin-Brown}
By Theorem \ref{Heller} and Example \ref{example-spectra}, every stable finitely presentable $\infty$-category satisfies Brown representability. 
\end{example}

We point out that since the representability problem for a functor $\C^{\op} \to \Set$ obviously reduces to the problem of representing the induced functor $\h \C^{\op} \to \Set$, it is also natural to state Brown representability theorems in terms of properties of $\h \C$ instead of $\C$. Indeed, Heller's theorem \cite[Theorem 1.3]{He} is such a theorem as it applies to general categories with coproducts, weak pushouts and a set of weak generators for which Proposition \ref{intermediate}(2) holds for some sufficiently large regular cardinal $\beta$ (in place of $\omega$). Moreover, there are also several 
well-known Brown representability theorems for triangulated categories (see, for example, \cite{Ne}).

The difference between the two viewpoints of $\C$ and $\h \C$ practically disappears in the case of compactly generated $\infty$-categories, essentially because of Proposition \ref{intermediate}. On the other hand, the proof of Proposition \ref{intermediate}(2) may fail for larger ranks of $h$-compactness because there are obstructions in general for lifting diagrams in $\h \C$ indexed by $\beta > \omega$ and for the colimit of such a diagram in $\C$ to define a weak colimit in $\h \C$ (see \cite{Co} for related results). As a consequence, it becomes unclear how properties of $\h \C$ which are required for Brown representability can be obtained from analogous properties of 
$\C$, since the former may turn out to be exotic from the viewpoint of the latter. 
In the next subsection we will obtain an extension of Example \ref{stable-fin-Brown} to general stable presentable $\infty$-categories indirectly, that is, without referring to Heller's theorem for the case of higher ranks of $h$-compactness. 

\subsection{Stable presentable $\infty$-categories} The following general structure theorem for stable presentable $\infty$-categories is an immediate 
consequence of \cite[Proposition 1.4.4.9]{HA} (and Example \ref{example-spectra}).

\begin{thm} 
Every stable presentable $\infty$-category is equivalent to a localization of a compactly generated stable $\infty$-category.
\end{thm}

Thus, combining this with Theorem \ref{Heller} and Proposition \ref{local-pres}, we obtain the following class of examples of $\infty$-categories which satisfy 
Brown representability. 

\begin{thm} \label{mainTHM}
Every stable presentable $\infty$-category satisfies Brown representabi\-lity.
\end{thm}

\begin{coro} \label{mainTHM-2}
Let $\mathcal{T}$ be a triangulated category which is equivalent to the homotopy category of a stable presentable $\infty$-category, as triangulated categories. Then a homological functor $F: \mathcal{T}^{\op} \to \mathrm{Ab}$ is representable 
if and only if it sends small coproducts in $\mathcal{T}$ to products.
\end{coro}
\begin{proof}
The condition is clearly satisfied by representable functors. For the converse, suppose that $\mathcal{T} \stackrel{\Delta}{\simeq} \h\C$, where $\C$ is presentable and stable. By Theorem \ref{mainTHM}, it suffices to show that the functor 
$$\C^{\op} \to \h\C^{\op} \simeq \mathcal{T}^{\op} \xrightarrow{F} \mathrm{Ab} \to \Set_*$$
satisfies (B1)--(B2). (B1) is satisfied by assumption. (B2) is an immediate consequence of the fact that $F$ is homological. 
\end{proof}

\end{document}